\documentclass[12pt]{article}
\usepackage [a4paper,left=2.5cm,right=2.5cm,top=2.5cm,bottom=2.5cm]{geometry}               
\usepackage{amssymb,amsmath,amsrefs,amsthm}
\usepackage{authblk}
\usepackage{abstract}

\providecommand{\keywords}[1]{\textbf{Keywords:} #1}
\usepackage{enumerate}
\usepackage{color} 
\usepackage{graphicx}
\usepackage{setspace}
\usepackage{csquotes} 
\usepackage{url}
\usepackage{hyperref}
\usepackage{booktabs}
\usepackage{arydshln}
\usepackage{multirow}
\usepackage{todonotes}

\hypersetup{colorlinks=true,
            linkcolor=blue,
            anchorcolor=blue,
            citecolor=blue}

\usepackage{caption}
\usepackage[final]{changes}

\usepackage{pifont}

\usepackage{subfigure}

\definechangesauthor[name={Reviewer1}, color=blue]{r1}
\definechangesauthor[name={Reviewer2}, color=cyan]{r2}

\newtheorem{theo}{Theorem}
\newtheorem{coro}{Corollary}
\newtheorem{lemma}{Lemma}

\newtheorem{prop}{Proposition}

\theoremstyle{definition} 
\newtheorem{rema}{Remark}

\title{Symmetric Central Configurations in the Concave 4-Body Problem with Two Pairs of Equal Masses}
\author[1]{Yangshanshan Liu}
\affil[1]{Chern Institute of Mathematics, Nankai University, Tianjin, 300071, China, lyss6133@nankai.edu.cn}
\author[2]{Zhifu Xie\thanks{The corresponding author is partially supported by Wright W. and Annie Rea Cross Endowment Funds at the University of Southern Mississippi and by National Science Foundation (Award Number (FAIN): 2447675).}\thanks{$^*$Corresponding author: Zhifu Xie}}
\affil[2]{School of Mathematics and Natural Science, The University of Southern Mississippi, Hattiesburg, MS 39406, USA, Zhifu.Xie@usm.edu}

\date{March 26, 2026 }  
\begin{document}
\maketitle
\begin{abstract}
We establish the existence of a single-parameter family of the concave kite central configurations in the 4-body problem with two pairs of equal masses. 
In such configurations, one pair of masses must lie on the base of an isosceles triangle, and the other pair on its symmetric axis with one mass positioned inside the triangle formed by the other three. Using a rigorous computer-assisted analytical approach, we prove that for any non-negative mass ratio, the number of such configurations is either zero, one, or two, thereby providing a complete classification of this family. Furthermore, we show that the unique configuration corresponding to a specific mass ratio is a fold-type bifurcation point within the reduced subspace. We also give a clear and complete bifurcation picture for both symmetric and asymmetric cases of this concave type across the entire planar 4-body configuration space.
\end{abstract}


\tableofcontents
\section{Introduction}

The classical $n$-body problem concerns the study of $n$ particles with positive masses $m_i$ and positions $q_i=(q_{i1},\cdots,q_{id})^T \in \mathbb{R}^{d}\,(i=1,\cdots,n;\, d=2,3)$ governed by Newton's laws of motion
\begin{equation}\label{newton}
	m_i\ddot q_{ij} =\frac{\partial U}{\partial q_{ij}},\quad i=1,\cdots,n,\,j=1,\cdots,d,
\end{equation}
where 
\begin{equation*}
	U=\sum\limits_{i<j}\frac{m_im_j}{\vert q_j-q_i \vert}=\sum\limits_{i<j}\frac{m_im_j}{r_{ij}}
\end{equation*}
is the Newtonian potential and $r_{ij}=|q_i-q_j|$ is the mutual distance between the bodies $m_i$ and $m_j$. 

This set of equations admits a class of particular solutions called homographic solutions, namely, each body moves on a similar Keplerian orbit, and the configuration remains the same up to symmetry. This particular position $q=(q_1,\cdots,q_n)$ of the $n$ bodies is called a central configuration and a mathematical expression is shown as follows 
\begin{equation}\label{cc}
	-\lambda m_j(q_j-c)=\frac{\partial U(q)}{\partial q_j}=\sum_{i\neq j,i=1}^{n}\frac{m_im_j(q_i-q_j)}{r_{ij}^3},
\end{equation}
where $c=\frac{1}{\sum m_i}\sum \limits_{i=1}^n m_iq_i$ is the center of mass and $\lambda$ the configuration constant. The equations \eqref{cc} are invariant under translation, dilation, and rotation, and we say two central configurations are equivalent if they can be transformed from one to another by combining the above three symmetric operations. 
 
The first five central configurations are found by Euler and Lagrange, corresponding to three collinear central configurations and two equilateral triangle central configurations up to symmetry, respectively. These central configurations are now named after them.
Similar questions are raised for the general case, which is known as the finiteness conjecture \cites{chazy1918, smale1998, wintner1941}: For any given $n$ positive masses, the number of equivalent classes of central configurations of the $n$-body problem is finite. 
The conjecture holds for $n=4$ \cite{hampton2006}, and for $n=5$, only generic results have been derived \cites{albouy2012,hampton2011}. For $n=6$, despite the new progress that has been made in \cite{chang2024}, it remains considerably distant from achieving the generic result.
Less is known for $n\geq7$. It also holds for some very particular shapes, such as the collinear central configurations \cite{moulton1910} for $n$ bodies, and some equal mass cases \cites{albouy1996,lee2009, moczurad2019}.

A non-collinear central configuration of the $n$-body problem is \textit{convex}, if each body lies on the boundary of the convex hull formed by all the $n$ bodies. It is \textit{strictly convex} if each body lies in the exterior of the convex hull formed by the other $n-1$ bodies. It is \textit{concave} if it is not strictly convex, and it is \textit{strictly concave} if it is not convex.
In the 4-body problem, except for the collinear cases, there is no central configuration that is convex (resp., concave) but not strictly convex (resp., concave), according to the Perpendicular Bisector Theorem (or Lemma \ref{perpendicular} in this paper). 
Despite a rich body of research, a complete and precise classification of the 4-body central configurations remains unavailable.
The convex case, however, is comparatively better understood. MacMillan and Bartky established the existence of a convex central configuration for any four ordered masses \cite{macmillan1932}, with an alternative proof provided later by Xia \cite{xia2004}. Regarding the uniqueness aspect, while Simó and Yoccoz addressed the problem, a rigorous mathematical proof is still lacking, with support limited to numerical evidence. The existence and the uniqueness conjecture has been confirmed for several specific convex shapes \cites{cors2012, santoprete2021a, santoprete2021b, xie2012}. Sun, Xie, and You \cite{sun2023} recently made some progress, giving a proof for masses sufficiently bounded away from zero. 
The characterization of the convex case with partial equal masses has been thoroughly established \cites{albouy2008,fernandes2017,perez2007}, and Corbera, Cors, and Roberts have achieved a complete classification of the convex 4-body central configurations in \cite{corbera2019}. 

For the concave 4-body central configurations, the classification remains unresolved, despite the proof of their existence \cite{hampton2002}. The solutions for the concave shapes with some partial equal masses were mentioned or discussed in \cites{martha2013,corbera2014,long2003,rusu2016,shi2010}. 
In addition, a simple degenerate concave central configuration in the 4-body problem was discovered by Palmore \cite{palmore1975b}, which facilitated studies of central-configuration bifurcations \cites{bernat2009,leandro2003,meyer1988,rusu2016,simo1978}. 

A non-collinear central configuration is called a \textit{kite} shape if it has a symmetry axis passing through two of the four masses. 
If a convex central configuration has two opposite equal masses, it must be a kite \cites{albouy2008,long2002}. Furthermore, it is actually a rhombus according to the study by Bernat, Llibre, and P\'{e}rez-Chavela. For the kite shape with three equal masses, Bernat et al. \cite{bernat2009} gave a full classification. Corbera and Llibre \cite{corbera2014} gave a complete description of the families of central configurations with two pairs of equal masses and two equal masses sufficiently small. They found five of non-collision types and three collision types, including both symmetric and asymmetric cases.
Santoprete characterized the kite shape in terms of mutual distances. Roberts \cite{roberts2025} showed the existence and uniqueness of the convex kite shape, and discussed the corresponding bifurcation and linear stability. 

In this paper, we focus on the concave central configurations of the 4-body problem. 
We present a nice characterization of the kite shape with two pairs of equal masses. 
\begin{prop}\label{proposition1}
Suppose that a concave kite central configuration for the 4-body problem with two pairs of equal masses possesses mass values $1$ and $m$. Then one pair of the equal masses must lie on the base of an isosceles triangle, and the other pair lies on the symmetry axis of the base, with one mass inside the triangle formed by the other three.
\end{prop}
One can derive this Proposition from the Dziobek equations for the planar 4-body problem without much difficulty. We also provide a more detailed proof in Section \ref{prop1} under the basic settings in this paper.
Then, we establish analytically in Theorem \ref{theo1} that such concave kite central configurations constitute a single-parameter family. 
Then we employ interval arithmetic and the Krawczyk operator to establish a rigorous, computer-assisted proof for the complete classification of this concave kite configuration, as presented in Theorem \ref{theo2}.

\begin{theo}\label{theo1}
Suppose that the planar 4-body central configuration is a concave kite shape with two pairs of equal masses. Then, one pair of the two must lie on the base of an isosceles triangle, and the other pair lies along the symmetry axis, with one body inside the triangle formed by the other three. Together, these concave kite central configurations constitute a one-parameter family. 
\end{theo}
\begin{rema}
	Roughly speaking, geometrically, this family of central configurations is fully \enquote{determined} by the small blue isosceles triangle, which is uniquely determined by $r_{13}$ if the base $r_{12}$ is fixed, as shown in Figure \ref{isotri}. Specifically, this small triangle deforms continuously from a degenerate triangle position with $m_1,m_2$ and $m_3$ collinear, to an equilateral triangle position. Correspondingly, the angle $\angle 312=\angle 321$ increases from $0$ to $\pi/3$.
\end{rema}

\begin{theo}\label{theo2}
Suppose that the concave kite central configurations possess two pairs of equal masses, with $m_1=m_2=1$ on the base and $m_3=m_4=m>0$ on the symmetry axis. Then 
\begin{enumerate}
	\item there are precisely two types of concave kite central configurations if $m\in (0,m_0)$; 
	\item there is only one if $m=m_0$; 
	\item there is no such kite central configuration if $m>m_0$, 
\end{enumerate}
where $m_0\approx1.002713329037083$. 

Furthermore, if $m=0$ is admissible, there are exactly two types of concave kite central configurations. In both, $m_1,m_2$ and $m_4$ form an equilateral triangle convex hull, and $m_3$ is located either on the midpoint of $m_1$ and $m_2$, or it coincides with $m_4$. 
\end{theo}
\begin{figure}[htbp]
        \begin{minipage}{0.5\textwidth}
        \centering
        \includegraphics[width=\textwidth]{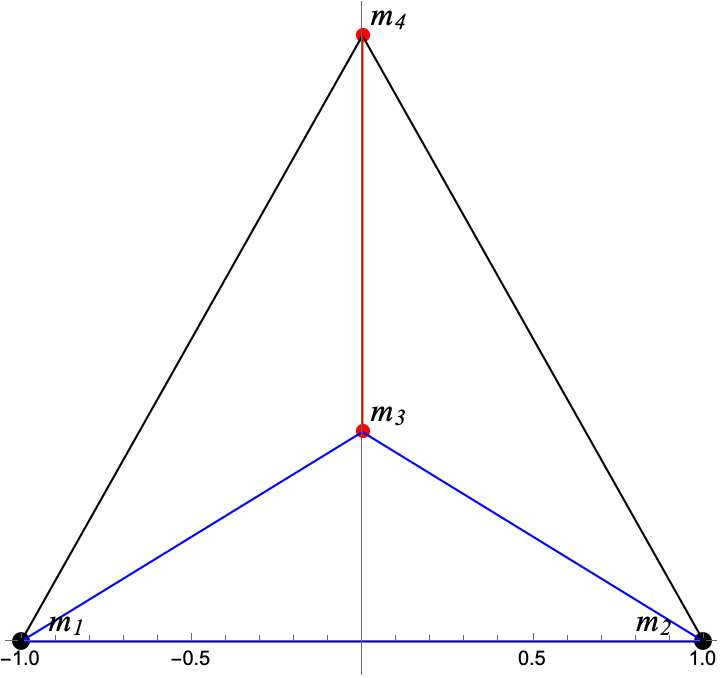}
    \end{minipage}
      \hfill
    \begin{minipage}{0.45\textwidth}
        \centering
        \caption{A symmetric planar central configuration in the concave 4-body problem with two pairs of equal masses, where $m_1=m_2$ are located on the base and $m_3=m_4$ on its symmetric axis with $m_3$ inside the triangle formed by the other three.}
         \label{isotri}
    \end{minipage}
\end{figure}
Our results improve and extend those of Alvarez-Ram\'irez and Llibre  \cite{martha2013} for the same concave kite shape. While their study assumed $0<m<1$, we demonstrate that the configuration can be extended to $m=m_0>1$ and $m=0$. Similar discussions can be found in \cites{xie2012,sun2025} for the convex cases with two pairs of equal masses, i.e., the isosceles trapezoid and the rhombus central configurations, and in both cases, the mass ratio ranges in $[0,1]$ by normalizing the larger pair of masses to be 1. 

  
 Our study further examines the bifurcations of the concave configurations with two pairs of equal masses.
In \cite{rusu2016}, by using the Krawczyk method, Rusu and Santoprete studied the bifurcations of this type, rigorously proved the existence of the bifurcations, and classified them.
 In contrast to the analysis in their paper, Theorem \ref{theo2} establishes the critical threshold
$m_0$ beyond which the bifurcation branch ceases to continue. We also demonstrate that the unique solution at $m=m_0$ represents the only fold bifurcation point in this symmetric subspace, differing from the bifurcation structure in the entire planar configuration space of the 4-body problem. 
Building on the results from \cites{corbera2014,rusu2016}, as well as our study, we can establish a complete bifurcation picture for concave central configurations with two pairs of equal masses when the symmetry axis is fixed. The complete picture is illustrated in Figure \ref{bifurcation2}, with the corresponding numerical examples in Table \ref{values}. As the mass parameter $m$ increases from $0$ to $\infty$, symmetric solutions emerge in pairs. When $m$ passes through a pitchfork bifurcation point at $\tilde m_\ast\approx0.992299447752385$, pairs of asymmetric solutions appear. Finally, when $m$ reaches $m_0$, the two symmetric solutions coalesce and then vanish for $m>m_0$, leaving the asymmetric solutions as the only remaining ones.  

The whole paper is constructed as follows.
In the next section, we introduce some necessary results used in this paper, including the Perpendicular Bisector Theorem, basic theories and settings for interval arithmetic and the Krawczyk operator, and the related bifurcation theory. 
In Section \ref{iso2+2} we prove Theorem \ref{theo1} and \ref{theo2}. At the end of Section \ref{iso2+2}, we discuss the bifurcations of this concave central configuration with two pairs of equal masses.

\section{Theoretical Background}
\subsection{Perpendicular Bisector Theorem}
The following result is widely used to determine the possible positions of bodies in a planar central configuration. According to the concept raised by Conley and Moeckel in \cite{moeckel1990}, the mutual distance $r_{ij}$ between any two bodies, together with its perpendicular bisector, divides the plane into four quadrants. The union of the first and third quadrants is an hourglass-shaped region, which will be called a \enquote{cone}. Similarly, the second and fourth quadrants together form another cone. The phrase \enquote{open cone} refers to a cone minus the axes. One can see the colored region in Figure \ref{corollary1} (or Figure \ref{b<2}) for illustration.
\begin{lemma}[Perpendicular Bisector Theorem, \cite{moeckel1990}]\label{perpendicular}
Let $q$ be a planar central configuration and let $q_i$ and $q_j$ be any two of its points. Then, if one of the two open cones determined by the line through $q_i$ and $q_j$ and its perpendicular bisector contains points of the configuration, so does the other one.
\end{lemma}
\subsection{The interval arithmetic method and the Krawczyk operator}\label{IntK}
Suppose that $F(x):D\subset \mathbb{R}^N\to\mathbb{R}^N$ is a  $C^1$ function with $x\in\mathbb{R}^N$. We aim to solve the equation
\begin{equation}\label{equsys}
    F(x)=0.
\end{equation} 
We use interval analysis to find zeros of \eqref{equsys}. The basic theory, including the interval arithmetic and the Krawczyk operator, can be found in \cites{alefeld2000,krawczyk1969,moore2009,neumaier1990}, as well as its applications in central configurations \cites{moczurad2019,rusu2016,sun2023}. 
Let $[x]\subset \mathbb{R}^N$ be a closed interval. In other words, $[x]=[x_{11},x_{12}]\times [x_{21},x_{22}]\times \cdots [x_{N1},x_{N2}]$, where $x_{i1},x_{i2}(i=1,\cdots,N)$ are two endpoints of a closed interval in each dimension. If $x_{i1}=x_{i2}$ holds for all $i=1,\cdots,N$, $[x]$ denotes the singleton $\{(x_1,\cdots,x_N)\}$. We denote by $R(F;[x])= \{F(x):x\in [x]\}$ the range of $F$ over $[x]$ and $F([x])$ the final interval under the interval computation.
A simple but essential property of the interval arithmetic is \begin{equation}\label{intervalproperty}
   R(F;[x]) \subset F([x]).
\end{equation}
The Krawczyk operator is based on the modified Newton operator and the Brouwer fixed-point theorem.
The derivative, namely, the Jacobian matrix of $F$ at $x$, is denoted by $DF(x)$. If $DF(x)$ is non-singular, we define the Krawczyk operator of $F$ as the following:
\begin{equation}\label{kop}
    K(x_0,[x])=x_0-C\cdot F(x_0) +(Id-C\cdot DF([x])([x]-x_0)),
\end{equation}
where $C\in\mathbb{R}^{N\times N}$ is a linear isomorphism. The properties of this operator are listed as follows
\begin{lemma}\label{koperator}
\begin{enumerate}
\item If $x^\ast \in [x]$ and $F(x^\ast)=0$, then $x^\ast\in K(x_0,[x])$. 
\item If $K(x_0,[x])\subset int[x]$, then there exists exactly one solution of equation $F(x)=0$ in $[x]$. This solution is non-degenerate with $DF(x)$ an isomorphism.
\item If $K(x_0,[x])\cap [x]=\varnothing$, then for all $x\in [x]$ we have $F(x)\neq0$. 
\end{enumerate}
\end{lemma}

We can write a program in SageMath \cite{sagemath} to find zeros of \eqref{equsys}.  
The main idea of the program is to cut the $N$-dimensional interval into small enough pieces and find zero on each piece one by one. 
The property \eqref{intervalproperty} and the definition \eqref{kop}, together with Lemma \ref{koperator}, guarantee that this computer-assisted interval approach is rigorous in mathematics. This approach has been successfully applied in \cite{sun2023} to partially solve the existence and uniqueness conjecture for the planar convex 4-body central configurations, in which we can find very explicit and detailed information for the arithmetic. In contrast, our program is similar, but simpler and more efficient, since in this paper, $N$ equals 1 or 2, which makes it easy to iterate and calculate. We apply the program in Section \ref{iso2+2} in the proofs of Lemma \ref{g(a,b)} and Theorem \ref{theo2} to analyze the zeros of the corresponding equations, respectively. Then we use it again to discuss the bifurcations. 

\subsection{Bifurcations}\label{bifurcations}
Suppose that
\begin{equation}\label{bifurF}
\begin{aligned}
    F:\mathbb{R}^N\times \mathbb{R}&\to \mathbb{R}^N\\
     (x,\mu)&\mapsto F(x,\mu),x\in \mathbb{R}^N,\mu\in\mathbb{R}
\end{aligned}
\end{equation}
is a smooth map and $\mu$ is a parameter. We follow the notation in Subsection \ref{IntK}, that $DF(x)$ is the Jacobian of $F$ with respect to $x$. We use $F_\mu$ to denote the vector of partial derivatives of the components of $F$ with respect to $\mu$. Consider the system $F(x,\mu)=0$.  When $\mu$ varies, the number of its solutions may change, implying bifurcations in the system. 
The following theorem provides sufficient conditions for certain bifurcations. One can refer to \cite{kuznetsov2004} for more information, as well as to \cite{rusu2016} for a particular application in central configurations. 
\begin{theo}\label{typeofbifur}
    Suppose that $F(x_0,\mu_0)=0$, and the Jacobian matrix $J_F=DF(x_0,\mu_0)$ has a simple eigenvalue $\lambda_0=0$ with eigenvector $v$. Denote by $w$ the adjoint eigenvector, i.e., $J_F^T w=0$, where $J_F^T$ is the transpose of $J_F$. 
\begin{enumerate}
    \item If $\left\{\begin{aligned}
            w^TF_\mu(x_0,\mu_0)&\neq0\\
            w^T(D^2F(x_0,\mu_0)(v,v))&\neq0\end{aligned}\right.,$ \\then the system experiences a \textbf{fold bifurcation} at the equilibrium point $x_0$ as the parameter $\mu$ passes through the bifurcation value $\mu=\mu_0$.
    \item If $\left\{\begin{aligned}
            w^TF_\mu(x_0,\mu_0)&=0\\
            w^T(DF_\mu(x_0,\mu_0)v)&\neq0\\
            w^T(D^2F(x_0,\mu_0)(v,v))&\neq0\end{aligned}\right.,$ \\then the system experiences a \textbf{transcritical bifurcation} at the equilibrium point $x_0$ as the parameter $\mu$ passes through the bifurcation value $\mu=\mu_0$.
    \item If $\left\{\begin{aligned}
            w^TF_\mu(x_0,\mu_0)&=0\\
            w^T(DF_\mu(x_0,\mu_0)v)&\neq0\\
            w^T(D^2F(x_0,\mu_0)(v,v))&=0\\
            w^T(D^3F(x_0,\mu_0)(v,v,v))&\neq0\end{aligned}\right.,$ \\then the system experiences a \textbf{pitchfork bifurcation} at the equilibrium point $x_0$ as the parameter $\mu$ passes through the bifurcation value $\mu=\mu_0$. \\If $w^T(D^3F(x_0,\mu_0)(v,v,v))<0$, the branches occur for $\mu>\mu_0$, and the bifurcation is \textbf{supercritical}. Otherwise, the branches occur for $\mu<\mu_0$, and the bifurcation is \textbf{subcritical}.
\end{enumerate}
\end{theo}

\section{Main results}\label{iso2+2}

We show in the following that the only position of a concave kite central configuration with two pairs of equal masses is the one shown in Figure \ref{isotri}. In other words, one pair of the equal masses must lie on the base of an isosceles triangle, and the other pair lies on the symmetry axis of the base, with one mass inside the triangle formed by the other three.
Let 
$$r_{13}=a,r_{14}=b$$ 
and $r_{12}=2$. 
The central configuration equations can be written under the position coordinates as
\begin{equation}\label{originalcceq}
\left\{
	\begin{aligned}
		\lambda  (q_1-c)=&\frac{m_2 (q_1-q_2)}{r_{12}^3}+\frac{m_3 (q_1-q_3)}{r_{13}^3}+\frac{m_4 (q_1-q_4)}{r_{14}^3},\\
\lambda  (q_2-c)=&\frac{m_1 (q_2-q_1)}{r_{12}^3}+\frac{m_3 (q_2-q_3)}{r_{23}^3}+\frac{m_4 (q_2-q_4)}{r_{24}^3},\\
\lambda  (q_3-c)=&\frac{m_1 (q_3-q_1)}{r_{13}^3}+\frac{m_2 (q_3-q_2)}{r_{23}^3}+\frac{m_4 (q_3-q_4)}{r_{34}^3},\\
\lambda  (q_4-c)=&\frac{m_1 (q_4-q_1)}{r_{14}^3}+\frac{m_2 (q_4-q_2)}{r_{24}^3}+\frac{m_3 (q_4-q_3)}{r_{34}^3}.
	\end{aligned}
	\right.
\end{equation}
We set 
\begin{equation}\label{q1234}
	q_1=(-1,0), q_2=(1,0),q_3=(0, \sqrt{a^2-1}),q_4=(0,\sqrt{b^2-1})
\end{equation}
with $b > a>1$. 
\subsection{Proof of Proposition \ref{proposition1}}\label{prop1}
The proposition holds for $m=1$ according to the two concave equal mass central configurations in \cite{albouy1996}, i.e., one possesses an equilateral triangle convex hull, and the other one possesses an isosceles triangle shape.

If $m\neq 1$, we suppose that $m_1=m_4=1$ and $m_2=m_3=m$. Substituting the masses and \eqref{q1234} into \eqref{originalcceq} and after some simplification we obtain six equations
\begin{equation}
\left\{
\begin{aligned}
		0=&\frac{ m}{a^3}+\frac{1}{b^3}-\frac{ \lambda  (3 m+1)}{2(m+1)}+\frac{m}{4},\\
		0=&-\frac{m}{a^3}-\frac{1}{b^3}+\frac{\lambda  (m+3)}{2 (m+1)}-\frac{1}{4},\\
		0=&\frac{m-1}{a^3}+\frac{\lambda(1 - m)}{2 (m+1)},\\
		0=&\frac{m-1}{b^3}+\frac{\lambda(1 - m)}{2 (m+1)},\\
		0=&-\frac{\lambda  \left(\sqrt{a^2-1} m+\sqrt{b^2-1}\right)}{2 (m+1)}+\frac{\sqrt{a^2-1} m}{a^3}+\frac{\sqrt{b^2-1}}{b^3},\\
		0=&\frac{\lambda  \left(\sqrt{a^2-1}-\sqrt{b^2-1}\right)}{m+1}+\frac{1}{\left(\sqrt{a^2-1}-\sqrt{b^2-1}\right)^2}-\frac{\sqrt{a^2-1}}{a^3}+\frac{\sqrt{b^2-1}}{b^3}.
		\end{aligned}
	\right.
\end{equation}
From the sum of the first two equations, we have 
$$\frac{(m-1) (-4 \lambda +m+1)}{4 (m+1)}=0,$$
which implies $\lambda =(m+1)/4$. Substituting it into the third and fourth equations, we have 
\begin{equation*}
	\begin{aligned}
		0=&-\frac{\left(a^3-8\right) (m-1)}{8 a^3},\\
		0=&-\frac{\left(b^3-8\right) (m-1)}{8 b^3},
	\end{aligned}
\end{equation*}
which implies $a=b=2$ since $m\neq 1$, contradicting to the assumption.
Almost in the same way, we derive the same contradiction when $m_1=m_3=1$ and $m_2=m_4=m$. Therefore, the only way to deposit the masses is $m_1=m_2$ and $m_3=m_4$, as shown in Figure \ref{isotri}.
\subsection{A single-parameter family}

Now let $m_1=m_2=1$ and $m_3=m_4=m$.
The center of mass is $c=(0,\frac{m(\sqrt{a^2-1}+\sqrt{b^2-1})}{2(1+m)})$. Substituting the masses and positions in \eqref{q1234} into \eqref{originalcceq} and with some simplification we obtain 
\begin{equation*}
	\left\{
	\begin{aligned}
		\lambda=& m (\frac{1}{a^3}+\frac{1}{b^3})+\frac{1}{4},\\
	0=&	\frac{\lambda   (\sqrt{a^2-1}+\sqrt{b^2-1})}{2 (m+1)}-\frac{\sqrt{a^2-1}}{a^3}-\frac{\sqrt{b^2-1}}{b^3},\\
	0=&-\frac{\lambda  m (\sqrt{a^2-1}+\sqrt{b^2-1})}{2 (m+1)}+\frac{m}{(\sqrt{a^2-1}-\sqrt{b^2-1})^2}+\sqrt{a^2-1}(\lambda -\frac{2}{a^3}),\\
	0=&-\frac{\lambda  m (\sqrt{a^2-1}+\sqrt{b^2-1})}{2 (m+1)}-\frac{m}{(\sqrt{a^2-1}-\sqrt{b^2-1})^2}+\sqrt{b^2-1} (\lambda -\frac{2}{b^3}).
	\end{aligned}
	\right.
\end{equation*}
Noticing that the second equation multiplied by 2 is equivalent to the sum of the last two equations. After multiplying the second equation by $m$ and adding it to the last two equations respectively, the first terms in both equations are eliminated. Then, we can substitute $\lambda$ derived from the first equation into the last two to get the following three equations: 
\begin{subequations}
\small
	\begin{align}
		&-\lambda +m (\frac{1}{a^3}+\frac{1}{b^3})+\frac{1}{4}=0,\label{l1}\\
		g_1=&m \left(\frac{1}{(\sqrt{a^2-1}-\sqrt{b^2-1})^2}+\frac{\sqrt{a^2-1}-\sqrt{b^2-1}}{b^3}\right)-\frac{ \sqrt{a^2-1}(8-a^3)}{4a^3}=0,\label{l2}\\
		g_2=&m \left(\frac{1}{(\sqrt{a^2-1}-\sqrt{b^2-1})^2}+\frac{\sqrt{a^2-1}-\sqrt{b^2-1}}{a^3}\right)+\frac{ \sqrt{b^2-1}(8-b^3)}{4b^3}=0.\label{l3}
	\end{align}
\end{subequations}
From \eqref{l2} we obtain 
\begin{equation}\label{mintwo}
	m=-\frac{\sqrt{a^2-1} (a^3-8) b^3 (\sqrt{a^2-1}-\sqrt{b^2-1})^2}{4 a^3 (b^3+(\sqrt{a^2-1}-\sqrt{b^2-1})^3)}.
\end{equation}
From \eqref{l2} and \eqref{l3} we obtain a mass-free equation
\begin{equation*}
\begin{aligned}
	&a^6 \sqrt{b^2-1} (b^3-8) (b^3+(\sqrt{a^2-1}-\sqrt{b^2-1})^3)\\
	+&b^6\sqrt{a^2-1} (a^3-8)  (a^3+(\sqrt{a^2-1}-\sqrt{b^2-1})^3)=0. 
\end{aligned}
\end{equation*}
First, we characterize $a$ and $b$ from Lemma \ref{perpendicular}. For convenience, we denote in the following the two open intervals respectively by 
\begin{equation*}
I_A=(1,2) \text{ and } \tilde I_B=(\sqrt{2},\sqrt{2}(\sqrt{3}+1)).
\end{equation*}

\begin{coro}\label{1<a<2}
	Suppose that $(a,b)$ is a solution for \eqref{l1}-\eqref{l3}, satisfying $a<b$. Then we have $a\in I_A$ and $b\in \tilde I_B$ for $m>0$. 
	More precisely, we have the following:
	\begin{enumerate}
		\item If $\sqrt{2}<b<2$, then $1<a<r_{34}<2/\sqrt{3}$.
		\item If $b=2$, then $a=2/\sqrt{3}$ and $m=1$.
		\item If $2<b<\sqrt{2}(\sqrt{3}+1)$, then $2/\sqrt{3}<r_{34}<a<2$. 
	\end{enumerate}
In addition, if $m=0$ is admissible, then $b=2$, and $a=1$ or $2$. 
\end{coro}
\begin{proof}
Firstly, by substituting $b=2$ into \eqref{l3}, we obtain $m=0$ or $a=2/\sqrt{3}$. Then by substituting $m=0$ and $a=2/\sqrt{3}$ into \eqref{l2} respectively, we get $a=1$ or $2$, and $m=1$, respectively.

Secondly, we claim that $a<2$ holds. Otherwise, suppose that $a>2$. According to Lemma \ref{perpendicular}, and as shown in Figure \ref{corollary1}, the red dashed line characterizes the equilateral triangle, and $A_1$ is the apex. Then $m_3$ is directly above $A_1$, and $m_4$ is directly above $m_3$. The blue dashed line is the perpendicular bisector of $r_{13}$. The blue region denotes one of the cones, divided by the two blue lines. We can clearly see that $m_4$ and $m_2$ will always lie in the blue region, which implies that it cannot be a central configuration when $a>2$. 

Thirdly, as shown in Figure \ref{b<2}, we assume that the center of the equilateral triangle is $A$. The intersection of $r_{14}$ and its perpendicular bisector is $A_2$, and the intersection of this perpendicular bisector and the symmetry axis is $A_3$. The right triangle $\triangle m_4A_2A_3$ is similar to $\triangle m_4Om_1$. Then we have 
$$\frac{\overline{m_4A_2}}{\overline{m_4O}}=\frac{\overline{m_4A_3}}{\overline{m_4m_1}},$$
which implies $\overline{m_4A_3}=b^2/(2\sqrt{b^2-1})$. 
\begin{itemize}
	\item When $b<2$, $A_3$ is below $A$. Then $m_2$ lies in the grey cone shown in Figure \ref{b<2}. This implies $m_3$ is below $A_3$, namely, i.e., $\overline{m_3O}<\overline{A_3O}<\overline{AO}$. From $\overline{m_3O}<\overline{AO}$, i.e., $\sqrt{a^2-1}<1/\sqrt{3}$, we have $r_{13}=a<\overline{Am_1}=2/\sqrt{3}$. From $\overline{m_3O}<\overline{A_3O}=\overline{m_4O}-\overline{m_4A_3}$, we have $\sqrt{a^2-1}<\sqrt{b^2-1}-b^2/(2\sqrt{b^2-1})$, i.e., $\sqrt{a^2-1}\sqrt{b^2-1}<(b^2-2)/2$, which implies $b>\sqrt{2}$.
	\item Similarly, when $b>2$, $m_2$ is not in the grey cone, and $\overline{m_3O}>\overline{A_3O}>\overline{AO}$, i.e., $\sqrt{a^2-1}>\sqrt{b^2-1}-b^2/(2\sqrt{b^2-1})>1/\sqrt{3}$. From the first and the third term, we have $a>2/\sqrt{3}$. The first inequality implies that $r_{13}>r_{34}$, i.e., $\sqrt{b^2-1}<a+\sqrt{a^2-1}<2+\sqrt{3}$. Then $b<\sqrt{2}(\sqrt{3}+1)$ with $a<2$.
\end{itemize}
\begin{figure}[htbp]
        \begin{minipage}[t]{0.5\textwidth}
        \centering
        \includegraphics[width=\textwidth]{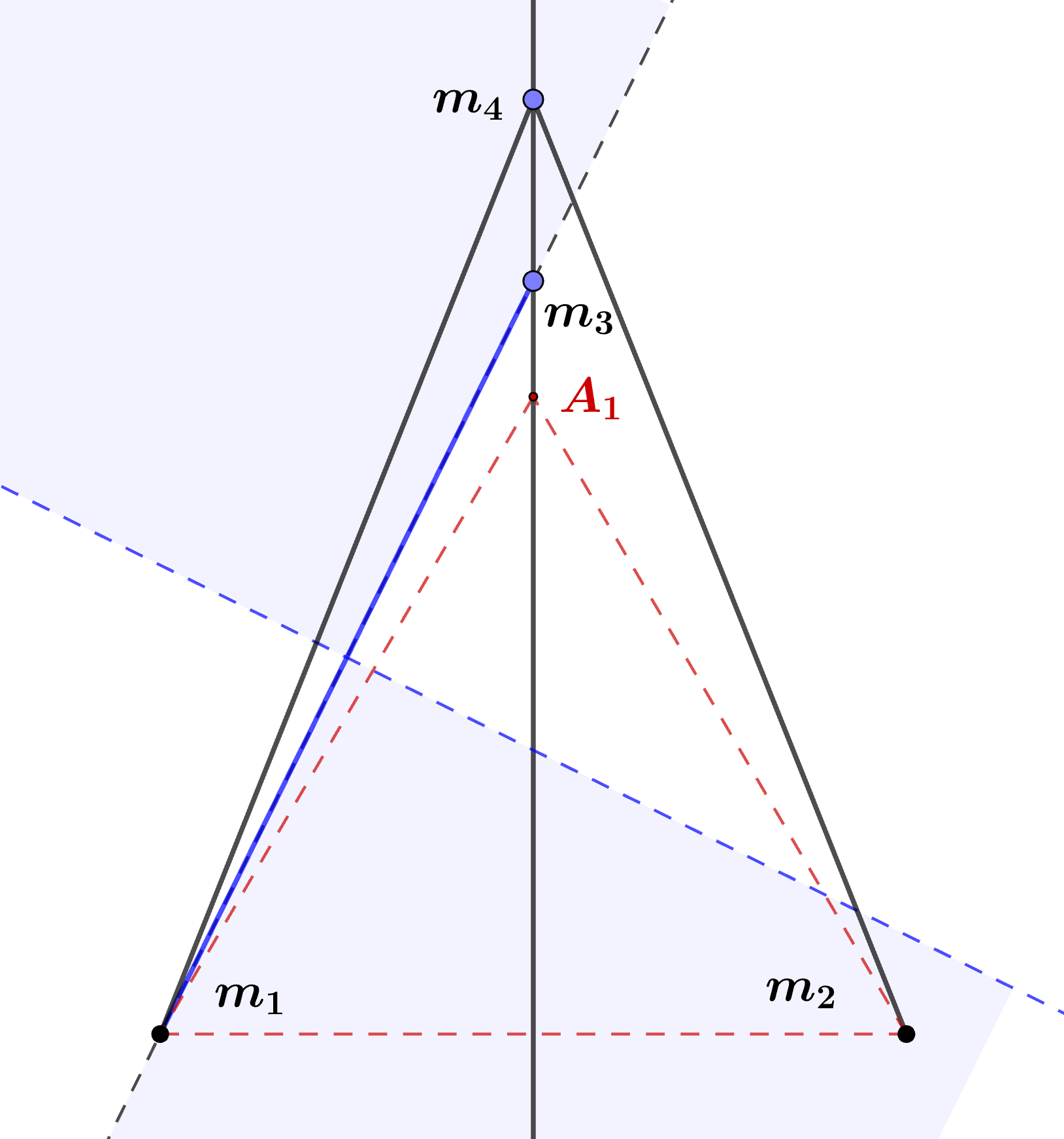}
        \caption{Bisecting $r_{13}$ when $r_{13}=a>2$.}
        \label{corollary1}
    \end{minipage}
      \hfill
    \begin{minipage}[t]{0.49\textwidth}
        \centering
        \includegraphics[width=\textwidth]{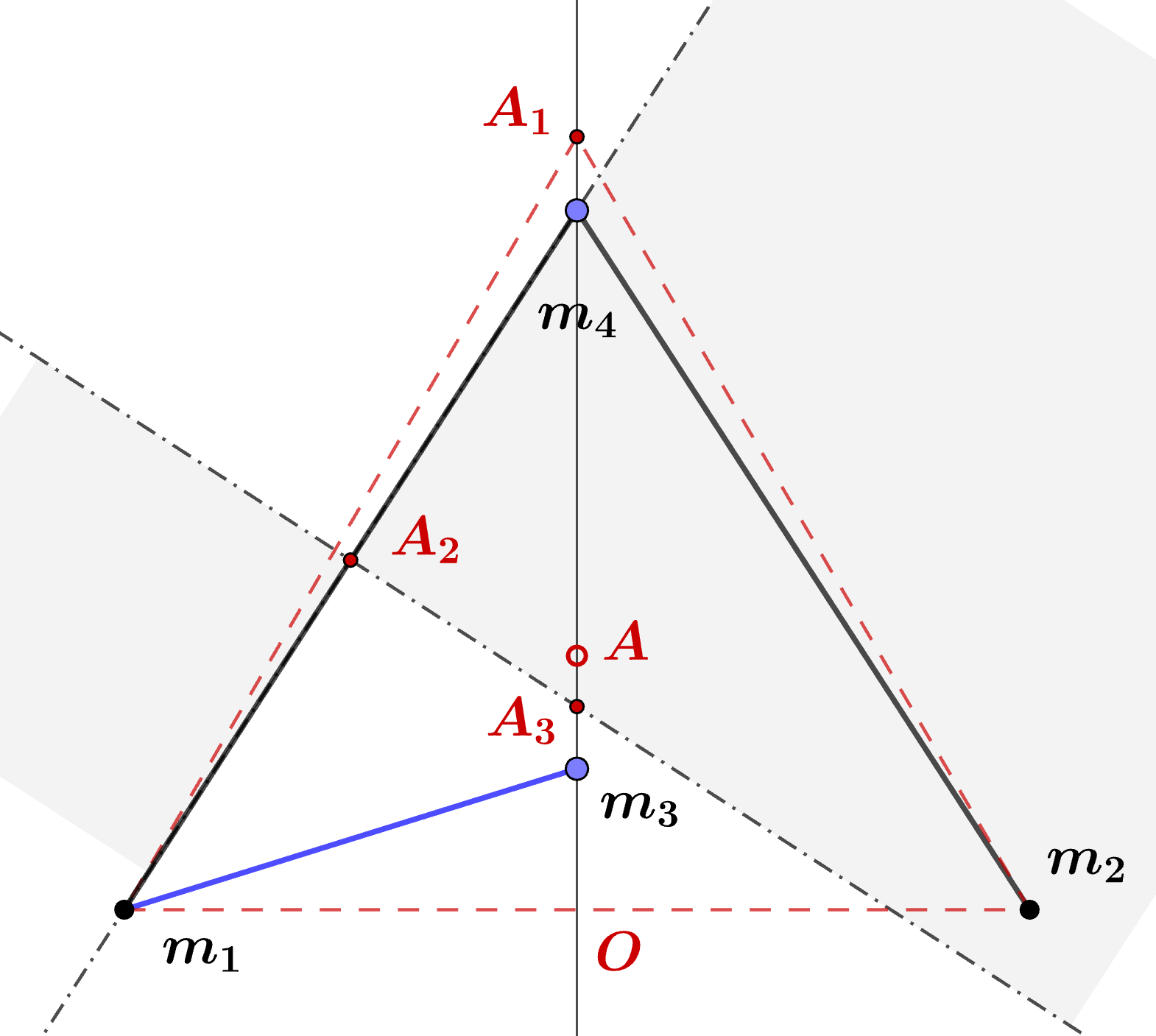}
        \caption{Bisecting $r_{14}$ when $r_{14}=b<2$.}
        \label{b<2}
    \end{minipage}
\end{figure}
\end{proof}
Now let 
\begin{equation}\label{xx1yy1}
\left\{
\begin{aligned}
	x=&\frac{\sqrt{a^2-1} (a^3-8)}{a^6},\\
	x_1=&(\sqrt{a^2-1}-\sqrt{b^2-1})^3+a^3,\\
	y=&\frac{\sqrt{b^2-1} (b^3-8)}{b^6},\\
	y_1=&(\sqrt{a^2-1}-\sqrt{b^2-1})^3+b^3.
\end{aligned}
	\right.
\end{equation}
Then $m$ in \eqref{mintwo} can be written as
\begin{equation}\label{m_in_xy}
m=\frac{w_1}{x_1}y,
\end{equation}
where $w_1=a^3b^4(\sqrt{a^2-1}-\sqrt{b^2-1})^2/4\geq0$. 
The mass-free equation can be written as 
\begin{equation}\label{g=xx1yy1}
xx_1+yy_1=0.
\end{equation}
We claim that the following inequalities hold.
\begin{lemma}\label{xy1}
Suppose that $x,x_1,y,y_1$ are defined in \eqref{xx1yy1}. Then
	\begin{equation}\label{x<0y1>0}
	x<0, \quad y_1>0,
\end{equation}
and 
\begin{equation}\label{partialsign}
\frac{\partial x_1}{\partial b}\leq0,\quad \frac{\partial (y\cdot y_1)}{\partial b}>0
\end{equation}
hold for any $a\in I_A$ and $b\in (\sqrt{2},b_1)\subsetneq \tilde I_B$, where $b_1$ is the only real root in $\tilde I_B$ of the polynomial equation 
\begin{equation}\label{num_of_dyb}
	p(z)=-2 z^5+3 z^3+40 z^2-48=0.
\end{equation}

\end{lemma}
\begin{proof}
For \eqref{x<0y1>0}, it is easy to see that $x<0$ when $1<a<2$. If $a\leq b$, then from the obtuse triangle $\triangle m_4m_1m_3$ shown in Figure \ref{isotri}, we find that $r_{14}$ is the longest side, which implies $r_{14}>r_{34}$, i.e., $y_1>0$. If $a>b$, $y_1>0$ holds obviously. 

For \eqref{partialsign}, by direct computation, we have 
\begin{equation*}
	\left\{
	\begin{aligned}
		\frac{\partial x_1}{\partial b}=&-\frac{3 b (\sqrt{a^2-1}-\sqrt{b^2-1})^2}{\sqrt{b^2-1}},\\
		\frac{\partial y_1}{\partial b}=& \frac{3 b(b\sqrt{b^2-1}-(\sqrt{a^2-1}-\sqrt{b^2-1})^2)}{\sqrt{b^2-1}},\\
		\frac{\partial y}{\partial b}=&\frac{b^2(b^3-8)+3b^3(b^2-1)-6(b^2-1)(b^3-8)}{b^7\sqrt{b^2-1}}=\frac{p(b)}{b^7\sqrt{b^2-1}}.
	\end{aligned}
	\right.
\end{equation*}
It is easy to see that $\partial x_1/\partial b\leq0$. For the second inequality, we have 
\begin{equation*}
	\frac{\partial (y\cdot y_1)}{\partial b} = \frac{\partial y}{\partial b} \cdot y_1+y\cdot\frac{\partial y_1}{\partial b}.
\end{equation*}
We consider three cases. 

Firstly, when $2<b<\sqrt{2}(\sqrt{3}+1)$, we have $y>0$. We aim to show that $\partial y/\partial b, \partial y_1/\partial b$ are both positive, and this implies $\partial (yy_1)/\partial b>0$. The expression in the numerator of $\partial y_1/\partial b$ satisfies $b\sqrt{b^2-1}-(\sqrt{a^2-1}-\sqrt{b^2-1})^2>\sqrt{b^2-1}^2-(\sqrt{a^2-1}-\sqrt{b^2-1})^2=\sqrt{a^2-1}(2\sqrt{b^2-1}-\sqrt{a^2-1})>0$. 
Meanwhile, for the numerator of $\partial y/\partial b$, without much difficulty, we compute the derivatives of $p(b)$ till the third order. $p'''(b)=18-120b^2<0$ in $(2,\sqrt{2}(\sqrt{3}+1))$, and this implies $p''(b)=80 + 18 b - 40 b^3$ is decreasing in $(2,\sqrt{2}(\sqrt{3}+1))$. We compute $p''(2)=-204$, which implies $p'(b)$ is decreasing in $(2,\sqrt{2}(\sqrt{3}+1))$. In addition, $p'(2.14)>0$ and $p'(2.15)<0$ implies $p(b)$ has a unique maximum in $(2.14,2.15)$. From $p(2.75)>0$ and $p(2.76)<0$ we find that $p(b)$ has a unique zero $b_1\in(2.75,2.76)$. So, for $2<b<b_1$, we have $\partial y/\partial b>0$. 

Secondly, for the case $\sqrt{2}<b<2$, we claim 
\begin{equation*}
\frac{\partial (y\cdot y_1)}{\partial b} \cdot(b^7\sqrt{b^2-1})= (\frac{\partial y}{\partial b} \cdot y_1+y\cdot\frac{\partial y_1}{\partial b})(b^7\sqrt{b^2-1})=k_1+k_2+k_3+k_4>0,
\end{equation*}
where
\begin{equation*}\label{k1234}
	\left\{
	\begin{aligned}
		k_1=&y_1\cdot (3-2b^2)(b^3-8),\\
		k_2=&y_1\cdot (3b^2-3)b^3,\\
		k_3=&y_1\cdot (3-3b^2)(b^3-8),\\
		k_4=&(b(b^2-1)-\sqrt{b^2-1}(\sqrt{b^2-1}-\sqrt{a^2-1})^2)\cdot 3b^2  (b^3-8).
	\end{aligned}
	\right.
\end{equation*}
It is obvious that $k_1$ and $k_2$ are positive. Furthermore we have
\begin{equation*}
	\begin{aligned}
		k_3+k_4=&3(b^3-8)[(-(b^2-1))(b^3-(\sqrt{b^2-1}-\sqrt{a^2-1})^3)+\\
		 &\qquad\qquad b^2\sqrt{b^2-1}(b\sqrt{b^2-1}-(\sqrt{b^2-1}-\sqrt{a^2-1})^2)]\\
		=&3(b^3-8)\sqrt{b^2-1}(\sqrt{b^2-1}-\sqrt{a^2-1})^2(\sqrt{b^2-1}^2-b^2-\sqrt{b^2-1}\sqrt{a^2-1})\\
		=&3(b^3-8)\sqrt{b^2-1}(\sqrt{b^2-1}-\sqrt{a^2-1})^2(-1-\sqrt{a^2-1}\sqrt{b^2-1})>0.
	\end{aligned}
\end{equation*}

For the case $b=2$, by direct computation we have $y=0$ and $\partial y/\partial b>0$, which implies $\partial (yy_1)/\partial b>0$. Together with the two cases above, we complete the proof.
\end{proof}

We denote by $g(a,b)=xx_1+yy_1$.

\begin{lemma}\label{g(a,b)}
For any $a\in I_A$, there exists a unique $b= \hat b(a)$ such that 
\begin{enumerate}
	\item $g(a,b)<0$ for $\sqrt{2}<b<\hat b(a)$,
	\item $g(a,\hat b(a))=0$, and 
	\item $g(a,b)>0$ for $\hat b(a)<b<5/2$.
\end{enumerate}
\end{lemma}
\begin{proof}
On the one hand, $g(a,\sqrt{2})=$
$$\frac{ (\sqrt{2}-4) ((\sqrt{a^2-1}-1)^3+2 \sqrt{2})}{4}+
\frac{(a^3-8) (a^3+(\sqrt{a^2-1}-1)^3) \sqrt{a^2-1}}{a^6}<0,$$
since the two parentheses in each term, i.e., $((\sqrt{a^2-1}-1)^3+2 \sqrt{2})>0$ and $a^3+(\sqrt{a^2-1}-1)^3>0$ hold with $a\in I_A$. 
On the other hand, $g(a,5/2)=$
$$\frac{244 \sqrt{21} ((\sqrt{a^2-1}-\sqrt{21}/2)^3+125/8)}{15625}+\frac{(a^3-8) \sqrt{a^2-1} (a^3+(\sqrt{a^2-1}-\sqrt{21}/2)^3)}{a^6}.$$ 
We apply the computer-assisted interval arithmetic introduced in Section \ref{IntK} to 
$$F=g(a,5/2)=0$$ 
to show rigorously that this single variable equation 
 has no zero when $a$ belongs to the closure of $I_A$, i.e., $ \overline{I_A}=[1,2]$. In addition, by direct computation, $g(3/2,5/2)=$
$$\frac{-15609375 \sqrt{5}+5558625 \sqrt{21}+23836392 \sqrt{105}-230181064}{22781250}>0,$$
which implies $g(a,5/2)>0$ for all $a\in \overline{I_A}$. Hence, by the Intermediate Value Theorem, there exists at least one
$b\in (\sqrt{2},5/2)$ such that $g(a, b) = 0$. 

On the other hand, from Lemma \ref{xy1}, for any $a\in I_A$ and $b\in (\sqrt{2},5/2)$ we have
\begin{equation*}
	\frac{\partial g}{\partial b}=x\cdot \frac{\partial x_1}{\partial b}+\frac{\partial (y\cdot y_1)}{\partial b}>0,
\end{equation*}
 since $5/2<b_1$. Hence, from the Implicit Function Theorem, $b=\hat b(a)$ is a differential function for $a\in I_A$. The set 
$$\Gamma=\{(a,b)\vert g(a,\hat b(a))=0,a\in I_A\}$$ 
is a smooth curve on the $(a,b)$-plane; see the black dotted-dashed curve in Figure \ref{g=0}. 
\end{proof}
\begin{lemma}\label{coro2}
	$g(a,b)=0$ has no zero in $\overline{ I_A}\times [5/2,\sqrt{2}(\sqrt{3}+1)]$. 
\end{lemma}
\begin{proof}
We split each one-dimensional interval into 100 equal parts and obtain $10^4$ subintervals. We check whether possible zeros exist in each subinterval. For example, for one of the subintervals $\theta=[1,1.01]\times [2.5,2.5+(\sqrt{2}(\sqrt{3}+1)-2.5)/100]$, we compute the interval value of $g(a,b)$ on this subinterval and obtain 
$$g([1,1.01],[2.5,2.5+(\sqrt{2}(\sqrt{3}+1)-2.5)/100])\subsetneq [0.232,11.623]\not\ni0.$$ 
Then, from the property \eqref{intervalproperty} we know that $0\notin R(g(a,b);\theta)$, which implies $g(a,b)$ has no zero on $\theta$. Similarly, we handle the remaining subintervals. Finally, we derive the result.
\end{proof}

\begin{figure}[htbp]
\centering
        \includegraphics[width=\linewidth]{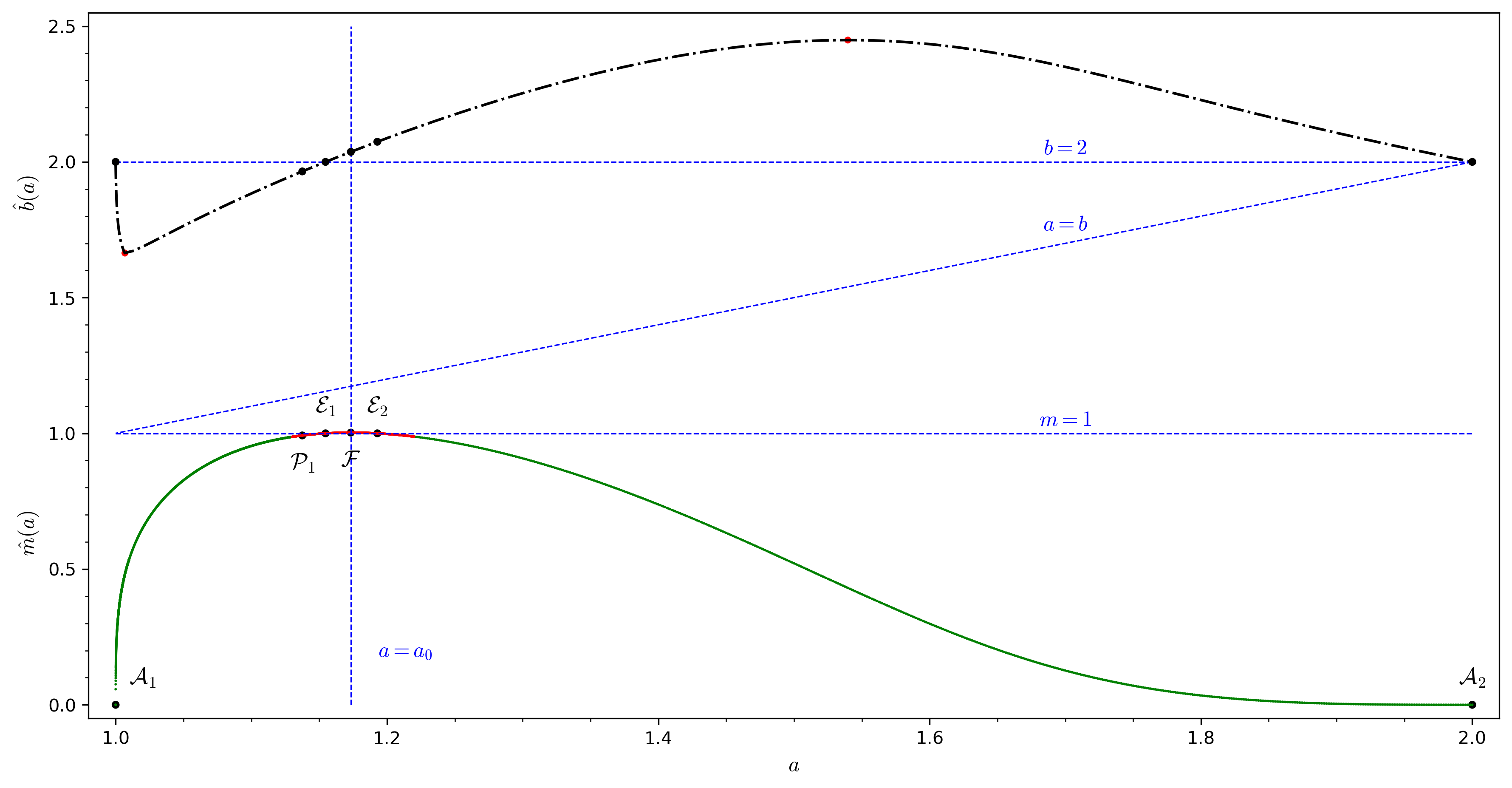}
        \caption{The black dashed-dotted curve represents the function $ b = \hat{b}(a) $, and the solid green curve represents $ m = \hat{m}(a) $, both defined for $ a \in I_A $. The minimum and maximum points of $ \hat{b}(a) $ are highlighted in red, located at approximately $ (1.0068269818055548, 1.6641309857549297) $ and $ (1.5397067078739939865, 2.4488397355312008965) $, respectively. The maximum point of $ \hat{m}(a) $, denoted $ (a_0, m_0) $, is approximately $ (1.1733802447932032924, 1.00271332903708271708) $. Specific central configurations, denoted by $\mathcal{A}_1,\mathcal{A}_2,\mathcal{P}_1,\mathcal{F},\mathcal{E}_1,\mathcal{E}_2$, corresponding to those in Figure \ref{bifurcation1}, Figure \ref{bifurcation2}, and Table \ref{values}, are labeled on the $ \hat{m}(a) $ curve with black dots. Their counterparts on the $ \hat{b}(a)$ curve are also indicated. An enlarged view of a portion of the green curve highlighted in red is presented in Figure \ref{bifurcation1}. The dashed blue lines are auxiliary lines, also labeled in the figure.}
        \label{g=0}
  \end{figure}
\begin{proof}[The proof of Theorem \ref{theo1}]
From Proposition \ref{proposition1} we know that for a symmetric concave kite central configuration with two pairs of equal masses, one pair should lie on the base, as shown in Figure \ref{isotri}.
From Lemma \ref{g(a,b)} we know that $b=\hat b(a)$ is uniquely determined by $a\in I_A$. 
By substituting $b=\hat b(a)$ into the expression of $m$ in \eqref{mintwo} we obtain
\begin{equation*}
m=m(a,\hat b(a))=\hat m(a)=-\frac{\sqrt{a^2-1} (a^3-8) \hat b^3 (\sqrt{a^2-1}-\sqrt{\hat b^2-1})^2}{4 a^3 ((\sqrt{a^2-1}-\sqrt{\hat b^2-1})^3+\hat b^3)}>0,
\end{equation*}
which implies that $\hat m(a)$ is a positive differential function with respect to the only variable $a\in I_A$. 

In conclusion, for any $a\in I_A$, $b=\hat b(a)$ and $m=\hat m(a)$ are uniquely determined successively. Hence, the concave kite central configurations in the planar 4-body problem with two pairs of equal masses, as shown in Figure \ref{isotri1}, form a single-parameter family with the parameter $a\in I_A$.
\end{proof}

\subsection{\texorpdfstring{The mass ratio $m=\hat m(a)$}{The mass ratio}}\label{massratio}
In this section, we prove Theorem \ref{theo2}. 
Using the computer-assisted interval arithmetic introduced in Section \ref{IntK}, we show rigorously that the smooth function $\hat m(a)$ possesses a unique maximum $m=m_0$ when $a\in I_A$. 
We use $I_B=(\sqrt{2},5/2)$ instead of $\tilde I_B$ according to Lemma \ref{coro2}. We denote by $ D=I_A\times I_B$.
\begin{proof}[The proof of Theorem \ref{theo2}]
Noticing that
$$\frac{d \hat m(a)}{da}=\frac{\partial m}{\partial a}-\frac{\partial  m}{\partial b}\cdot \frac{\frac{\partial g}{\partial a}}{\frac{\partial g}{\partial b}}.$$ 
We consider the equation $F=(f_1,g)=0$, where
\begin{equation}\label{uniquedm}
\left\{
	\begin{aligned}
		0=&\frac{\partial m}{\partial a}\cdot \frac{\partial g}{\partial b}-\frac{\partial m}{\partial b}\cdot \frac{\partial g}{\partial a}=f_1(a,b),\\
		0=&g(a,b).
	\end{aligned}
	\right.
\end{equation}
We aim to show $F(a,b)=0$ has a unique solution in the interior of $D$.
In fact, after excluding most of the interval pieces, we lock a small interval $D_0=[1.1,1.2]\times [2,2.1]$ containing possible zeros. The domain $D_0$ is first partitioned by dividing each dimensional interval into 100 equal parts. Each resulting subinterval is then evaluated with the Krawczyk operator from the interval arithmetic program, and all subintervals found to contain no zeros are excluded. Finally, we find the unique solution $(a_0,b_0)$ of \eqref{uniquedm} in
\begin{equation}\label{uniquesol}
\left\{
	\begin{aligned}
 a_0=&1.1733802447932032924212633685?\\
 b_0= &2.036986393189520563373721634?\\
	\end{aligned}
	\right.,
\end{equation}
and the corresponding mass is $m=m_0=1.00271332903708271708084572?$.
The question mark at the end denotes a standard floating-point number (with a small uncertainty) in SageMath.
In addition, from Lemma \ref{g(a,b)} we know that $\partial g/\partial b>0$ for $(a,b)\in D$, which implies that $d\hat m/da$ has the same sign as $f_1$. From Corollary \ref{1<a<2} we know that the point $(2/\sqrt{3},2)$ is a solution of $g(a,b)=0$. Then by direct computation, we have $f_1(2/\sqrt{3},2)=6(2\sqrt{3}-3)>0$, and this implies that $\hat m(a)$ increases in $(1,a_0)$ and decreases in $(a_0,2)$, since $2/\sqrt{3}<a_0$. Hence, $a_0$ is the unique maximum of $\hat m(a)$ in $I_A$. One can see the green curve in Figure \ref{g=0} for illustration. 

Furthermore, we consider the limits of $\hat m(a)$ when $a$ approaches the two endpoints of $I_A$.
Since $\lim\limits_{a\to 1^+}\hat b(a)=2$, and from the expression of $m$ in \eqref{m_in_xy}, we find $\lim\limits_{a\to1^+}x_1=1-3\sqrt{3}\neq0$. Similarly we have $\lim\limits_{a\to 2^-}\hat b(a)=2$ and $\lim\limits_{a\to2^-}x_1=8\neq0$. Then, by direct computation, we have 
\begin{equation*}
	\lim\limits_{a\to1^+} \hat m(a)=0,\quad \lim\limits_{a\to2^-} \hat m(a)=0, 
\end{equation*}
and this implies that $\hat m(a)$ can be continued to $0$ at $a=1$ and $a=2$ respectively, coinciding with the result in Corollary \ref{1<a<2}. In other words, $\hat m(a)$ continues on $\overline{I_A}=[1,2]$. 
Then the following claims hold directly
\begin{enumerate}
	\item When $\tilde m\in (0,m_0)$, the equation $\hat m(a)=\tilde m$ has two solutions, and there are exactly two concave kite central configurations with two pairs of equal masses. 
	\item When $\tilde m =m_0$, the equation $\hat m(a)=\tilde m$ has a unique solution, corresponding to a unique concave kite central configuration. 
	\item When $\tilde m >m_0$, no concave kite central configurations with two pairs of equal masses exists at all. 
    \item Furthermore, if $\tilde m=0$ is admissible, the equation $\hat m(a)=0$ has two solutions corresponding to $\mathcal{A}_1$ and $\mathcal{A}_2$ respectively, as shown in the green curve in Figure \ref{g=0}, which also correspond to the case $s_1$ and the case $sc_2$ respectively in \cite{corbera2014}.
\end{enumerate}
\end{proof}
Figure \ref{isotri1} shows the solutions for different $m$.
\begin{figure}[htbp]
    \centering
    \begin{minipage}{0.45\textwidth}
     \vspace{42pt}
        \includegraphics[width=\linewidth]{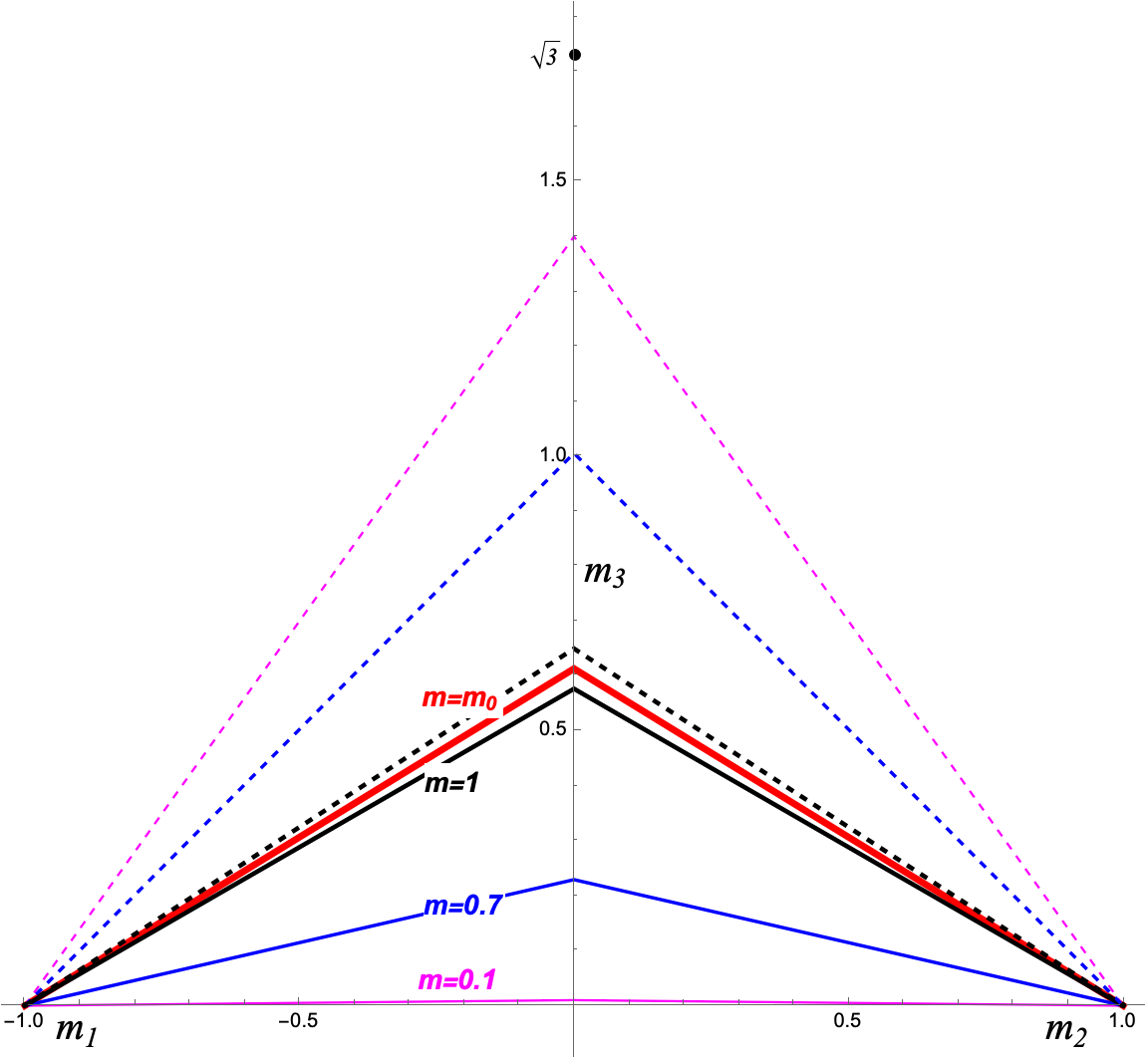} 
        \label{isosceles1}
    \end{minipage}
    \hfill 
    \begin{minipage}{0.477\textwidth}
      \includegraphics[width=\linewidth]{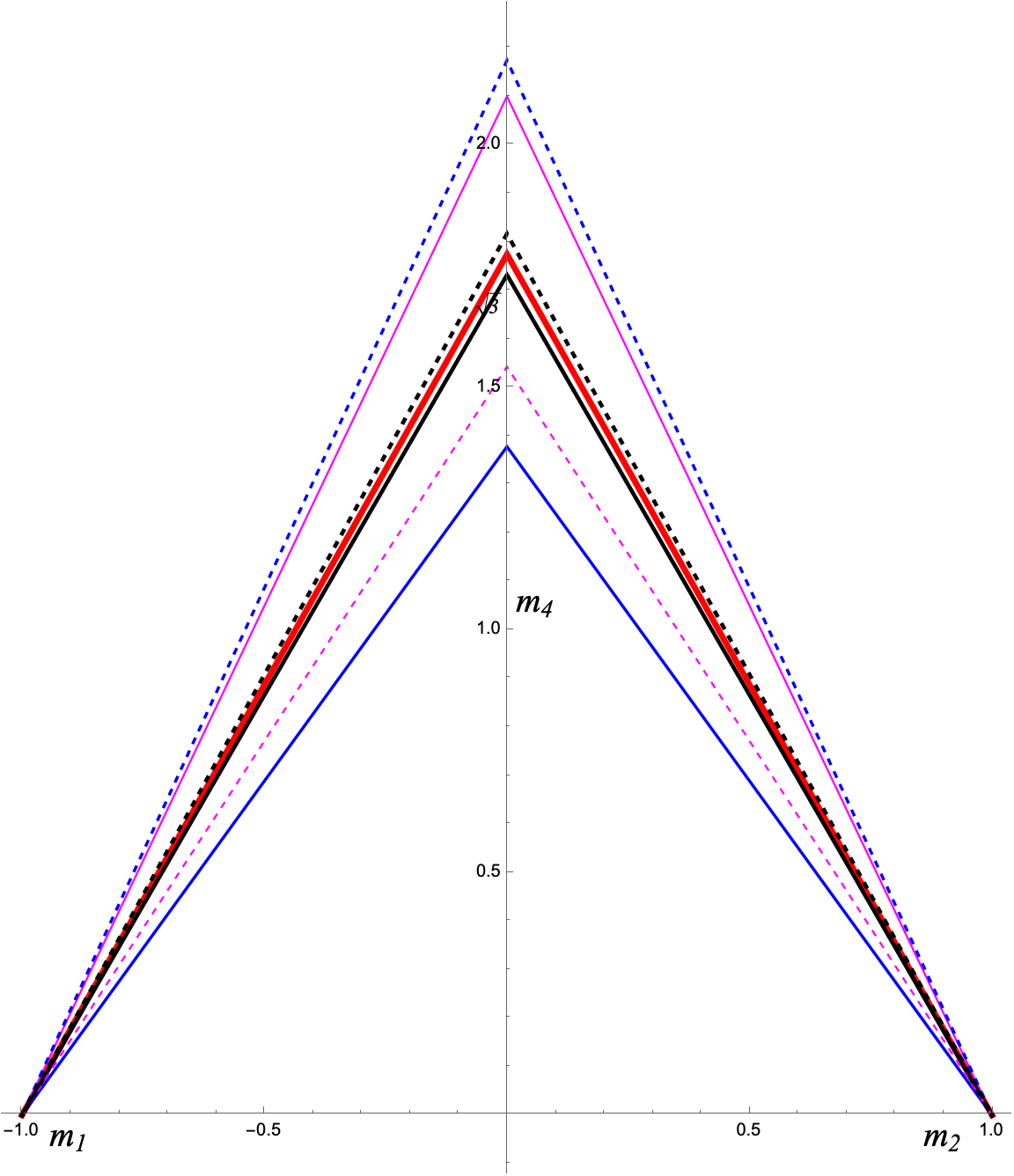} %
        \label{isosceles2}
    \end{minipage}
     \caption{Some examples of the concave kite central configurations with two pairs of equal masses. The left and right figures show the positions of masses $m_3$ and $m_4$, respectively. In the left figure, different colors represent different mass values of $m_3$, and the dashed and solid lines of the same color denote the two distinct solutions, respectively.
      The right figure uses the same color and line scheme to denote the corresponding positions of $m_4$ for each mass value. For example, the red lines correspond to the unique solution where $m = m_0$. The black dashed and solid lines denote the two distinct solutions for the equal mass case, respectively.}
     \label{isotri1}
\end{figure}
For example, the black lines (solid and dashed) represent the two central configurations with four equal masses, where the solid lines denote an equilateral triangle and the dashed lines an isosceles triangle. The blue lines (both solid and dashed-dotted) represent the two cases with a mass ratio of $m=0.7$. The solid red lines describe the unique solution when $m=m_0$.

\subsection{Bifurcations in different spaces}\label{anabifur}
\subsubsection{Bifurcations in the reduced subspace}
In this section, we first discuss the bifurcation with the symmetric condition in our settings. We wonder about the possible bifurcations of the reduced equations \eqref{l2} and \eqref{l3} with respect to $r=(a,b)$ as the variable and $m$ as the parameter, i.e., 
\begin{equation*}\label{G=0}
	\tilde G(r,m)=(g_1,g_2)=0.
\end{equation*}
Now we rewrite $g_1=0$ and $g_2=0$ as $ g_3=m-h_1(a,b)=0$ and $g_4=m-h_2(a,b)=0$, where
\begin{equation*}
	\left\{
	\begin{aligned}
		h_1 (a,b)=& -\frac{b^3\sqrt{a^2-1} (a^3-8)  (\sqrt{a^2-1}-\sqrt{b^2-1})^2}{4 a^3 (b^3+(\sqrt{a^2-1}-\sqrt{b^2-1})^3)},\\
		h_2(a,b)=&\frac{a^3\sqrt{b^2-1} (b^3-8)  (\sqrt{b^2-1}-\sqrt{a^2-1})^2}{4 b^3 (a^3+(\sqrt{b^2-1}-\sqrt{a^2-1})^3)}.
	\end{aligned}
	\right.
\end{equation*}
We turn to consider the equation
\begin{equation*}
	G(r,m)=(g_3,g_4)=0.
\end{equation*}
Noticing that $\tilde r=(2/\sqrt{3},2)$ is a singularity, or a discontinuity point of $g_4$, i.e., the two iterated limits at $\tilde r$ are $m$ and $m-1/(2\sqrt{3})$ respectively. This implies $\tilde r$ is not a solution for $G=0$. We say the equations $\tilde G=0$ and $G=0$ are equivalent on $\hat D=D\backslash \{\tilde r\}$. Hence, the bifurcations of the two equations are the same.
The Jacobian of $G(r,m)$ with respect to $r$ is
\begin{equation*}
	J_G=\begin{bmatrix}\label{jc}
		j_{11}&j_{12}\\
		j_{21}&j_{22}
	\end{bmatrix}
	=
	\begin{bmatrix}
		\frac{\partial g_3}{\partial a}&\frac{\partial g_3}{\partial b}\\
		\frac{\partial g_4}{\partial a}&\frac{\partial g_4}{\partial b}
	\end{bmatrix}=-\begin{bmatrix}
		\frac{\partial h_1}{\partial a}&\frac{\partial h_1}{\partial b}\\
		\frac{\partial h_2}{\partial a}&\frac{\partial  h_2}{\partial b}
	\end{bmatrix}.
\end{equation*}
We denote by 
\begin{equation*}
	f_2(a,b)=\frac{\partial h_1}{\partial a}\cdot \frac{\partial h_2}{\partial b}-\frac{\partial h_1}{\partial b}\cdot \frac{\partial h_2}{\partial a},
\end{equation*}
Then the determinant of $J_G$ is $\det J_G=f_2$.
To find all the possible bifurcation points, we consider the equation 
\begin{equation}\label{f2g}
	H(r)=(f_2,g)=0.
\end{equation}

Firstly, we claim that the unique solution of the equation \eqref{uniquedm}, i.e., $r_0=(a_0,b_0)\in \hat D$ in \eqref{uniquesol} is also a solution for \eqref{f2g}. 
In fact, from $F(r_0,m_0)=0$, we have 
$$f_1(r_0)=0\Leftrightarrow \frac{dh_1}{da}\vert_{a=a_0}=0,$$
since $h_1(a,b)$ is just the expression of $m$ in \eqref{mintwo}, satisfying $h_1(a,\hat b(a))=\hat m(a)$.
This implies 
$$\frac{dh_2}{da}\vert_{a=a_0}=0,$$
since from \eqref{m_in_xy} and \eqref{g=xx1yy1} we have $h_1=w_1\cdot y/x_1\equiv-w_1\cdot x/y_1=h_2$. 
Hence
$$\frac{\frac{\partial h_1}{\partial a}}{\frac{\partial h_1}{\partial b}}=\frac{\frac{\partial g}{\partial a}}{\frac{\partial g}{\partial b}}=\frac{\frac{\partial h_2}{\partial a}}{\frac{\partial h_2}{\partial b}}$$
holds for $r=r_0$, which leads to $\det J_G(r_0)=f_2(r_0)=0$. 

Secondly, we apply the interval arithmetic program to \eqref{f2g} to find zeros on the domain $\hat D$. We note that to avoid the singularity during the calculation, we need to remove the denominator of $f_2$. In other words, we use 
\begin{equation}\label{f2'g}
	\tilde H(r)=(\tilde f_2,g)=0
\end{equation}
in the program on the closed interval $D$ to instead, where $\tilde f_2$ is the numerator of $f_2$. After excluding most of the interval pieces on $D$, we lock on a small piece near each zero. In the same way, we split each small interval into smaller enough pieces, with the long side of each rectangle no more than $10^{-8}$. Finally, we get the only two solutions for \eqref{f2'g}, and they are exactly  $r_0$ and $\tilde r$. One can see Figure \ref{3partial} for illustration. One of the solutions, namely, $\tilde r$, can be excluded directly, since it is not a solution for $G=0$. Therefore, $r_0$ is the only zero for \eqref{f2g}.
 \begin{figure}[htbp]
\begin{center}
\includegraphics[width=1\linewidth]{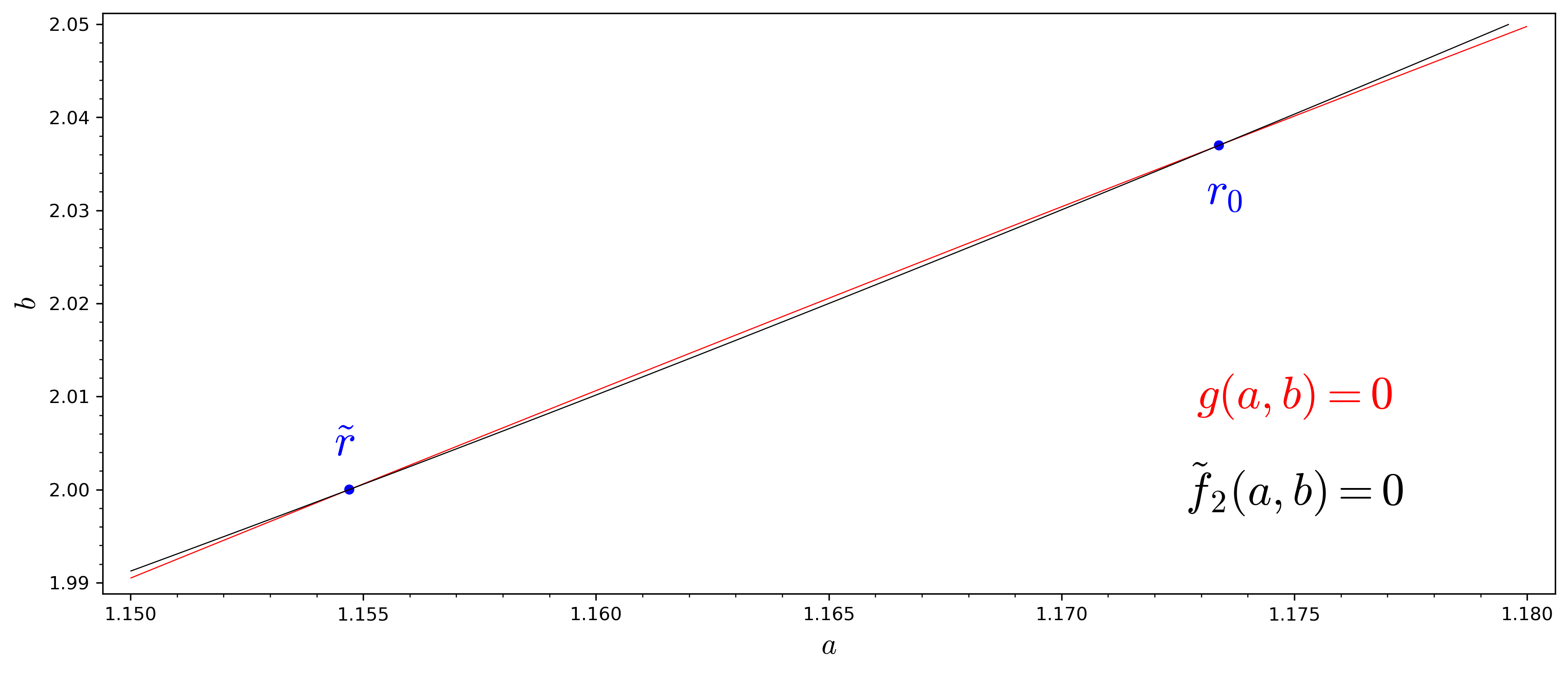}
\caption{The curves of $\tilde f_2(a,b)=0$ and $g(a,b)=0$ on the $(a,b)$-plane with $a\in [1.15,1.18]$. The only two intersections are $r_0=(a_0,b_0)$ and $\tilde r=(2/\sqrt{3},2)$.}
\label{3partial}
\end{center}
\end{figure}

With Theorem \ref{typeofbifur} in Section \ref{bifurcations}, we do the computations with interval arithmetic to determine the bifurcation type of $r_0$. 
The eigenvectors of the eigenvalue $0$ corresponding to $J_G$ and $J_G^T$ can be chosen as 
$$v=[-j_{12},j_{11}]^T \text{ and } w=[-j_{21},j_{11}]^T$$
respectively, since $J_G v=0$ and $J_G^Tw=0$. The partial derivatives of $G(r,m)$ with respect to $m$ is $G_m=[1,1]^T$. Substituting the value of $(r_0,m_0)$ into $J_G$ we have 
\begin{equation*}
	J_G(r_0,m_0)=\begin{bmatrix}
		4.5172058916474534136160228?&  -2.3231832749879904304948019?  \\
		231.5225618448226740599212?&
		-119.07124810379195739208953?
	\end{bmatrix}. 
\end{equation*}
Then, by direct computation, we obtain
\begin{equation*}
\left\{
	\begin{aligned}
		w^T\cdot G_m(r_0,m_0)=&-j_{21}(r_0,m_0)+j_{11}(r_0,m_0)\\
		=&-227.0053559531752206463052?\not\ni0,\\
		w^TD^2C(r_0,m_0)(v,v)=&w^T\cdot[D^2g_3(v,v),D^2g_4(v,v)](r_0,m_0)\\
		=&-18222.741196439596271921?\not\ni0,
	\end{aligned}
\right. 
\end{equation*}
This implies that $(r_0,m_0)$ is a bifurcation point of a fold type. When we perturb from the two equal mass points $\mathcal{E}_1$ and $\mathcal{E}_2$ in Figure \ref{bifurcation1} by increasing $m$ from 1, two symmetric central configurations coalesce into one. This solution cannot be continued further, since there is no symmetric central configuration for $m>m_0$, according to Theorem \ref{theo2}. 
 \begin{figure}[htbp]
\centering
\includegraphics[width=1\linewidth]{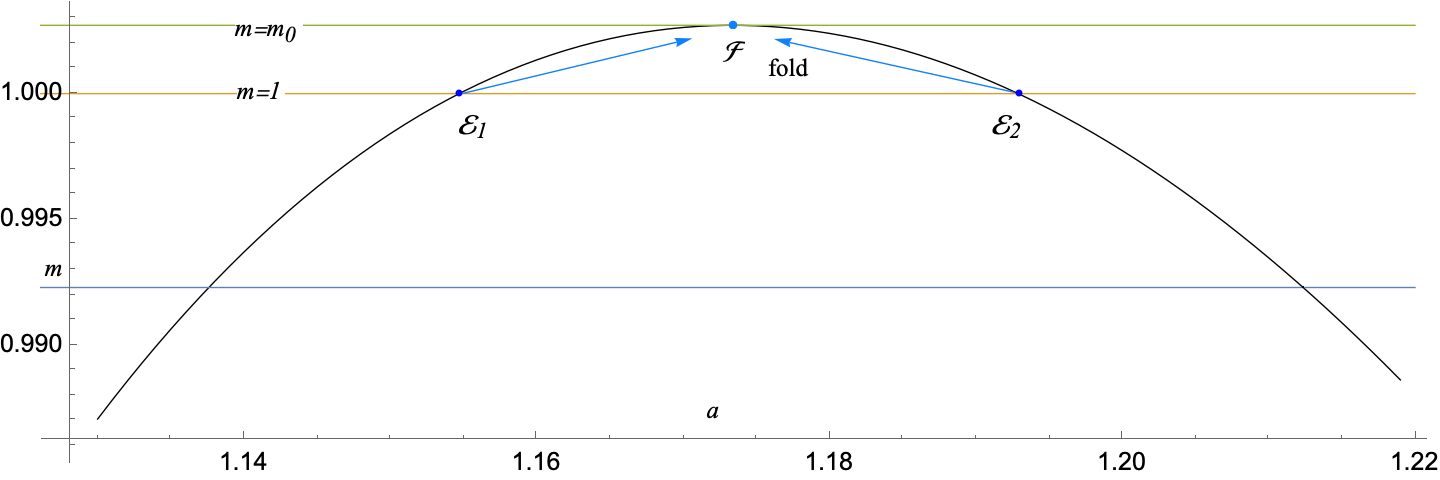}
\caption{We amplify a section on the curve $\hat m(a)$ around $m_0$, which is denoted in red of the green curve in Figure \ref{g=0}, to illustrate the bifurcations intuitively. By increasing $m$ from 1, $\mathcal{E}_1$ and $\mathcal{E}_2$ coalesce. The bifurcation occurs at $\mathcal{F}$ when $m$ reaches $m_0$. It is the only bifurcation point of the equation $G=0$. It is a fold type.}
\label{bifurcation1}
\end{figure}

 \begin{figure}[htbp]
\centering
\includegraphics[width=1\linewidth]{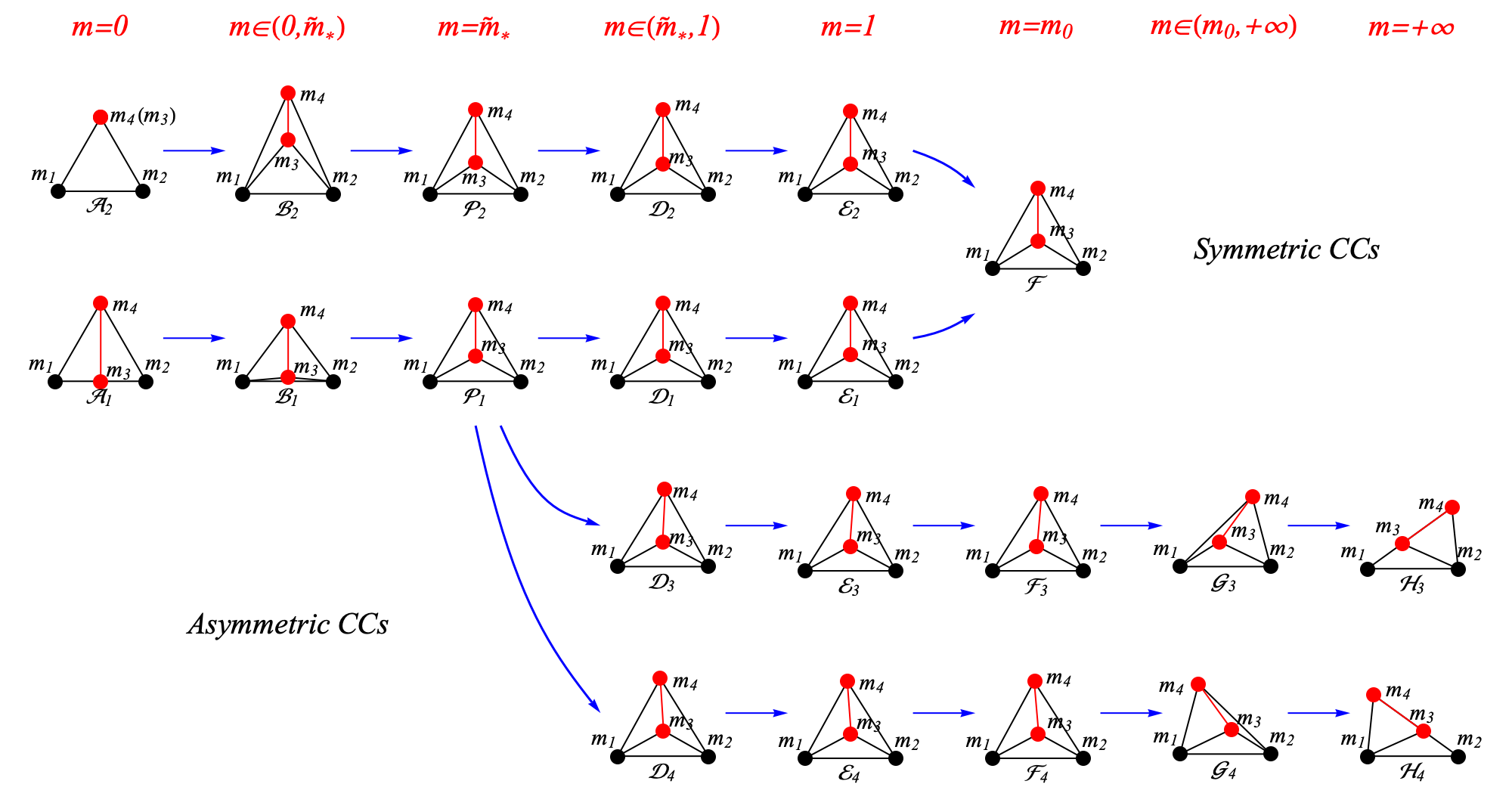}
\caption{Bifurcations of the concave 4-body central configurations with two pairs of equal masses, including both symmetric and asymmetric cases. We always set $m_1=m_2=1 $ and $ m_3=m_4=m $. With $m$ increasing from $0$ to $+\infty$, two bifurcation points are identified: at $\mathcal{P}_1$ when $m=\tilde m_\ast$, asymmetric configurations emerge, while at $\mathcal{F}$ when $m=m_0$, symmetric configurations cease to exist.}
\label{bifurcation2}
\end{figure}

 \subsubsection{Bifurcations in the whole planar 4-body configuration space}
  \counterwithin{table}{section}
\begin{table}[htbp]
\centering
\tiny
\begin{tabular}{@{}cc|c c@{}}
\toprule
\textbf{$m$} & \textbf{CCs} & \textbf{$q_3$} & \textbf{$q_4$} \\
\midrule
\multirow{2}{*}{0} & $\mathcal{A}_1$ & $(0,0)$ & $(0,\sqrt{3})$ \\
 &  $\mathcal{A}_2$ & $(0,\sqrt{3})$ & $(0,\sqrt{3})$\\
 \cline{1-4}
\multirow{2}{*}{$0.4$} & $\mathcal{B}_1$ & $(0, 0.08545589736279056172317148415)$ & $(0, 1.3373300780354451360145241818)$ \\
 &  $\mathcal{B}_2$ & $(0, 1.1886118075960024612098337898)$ & $(0, 2.2345238899447153200315721359)$ \\
\cline{1-4}
\multirow{2}{*}{$\tilde m_\ast$} &  \textcolor{blue}{$\mathcal{P}_1$} & $(0, 0.542337524224352560399170458)$ & $(0, 1.691445724232170729416786631)$ \\
 &  $\mathcal{P}_2$ & $(0, 0.685433331653046748968775972)$ & $(0, 1.857925542719253462435156388)$ \\
\cline{1-4}
\multirow{4}{*}{0.996} & $\mathcal{D}_1$ & $(0, 0.556434137865205678847558365)$ & $(0, 1.707753032824028314888992510)$ \\
 & $\mathcal{D}_2$& $(0, 0.671315754589387265540263247)$ & $(0, 1.841579560968450541481421594)$ \\
 & $\mathcal{D}_3$ & $(-0.009906700029766125, 0.5423813608539542)$ &$(0.04880076418676227, 1.6905216297662835)$ \\
 &$\mathcal{D}_4$ & $(0.009906700029766125, 0.5423813608539542)$ &$(-0.04880076418676227, 1.6905216297662835)$ \\
\cline{1-4}
\multirow{4}{*}{1} & $\mathcal{E}_1$& $(0,1/\sqrt{3})$ & $(0,\sqrt{3})$ \\
 & $\mathcal{E}_2$ & $(0, 0.6503784729520715)$ & $(0, 1.817239394723845)$ \\
 & $\mathcal{E}_3$ & $(-0.014277689766976964, 0.5424284291298985)$ & $(0.07027749578541109, 1.6895283608200573)$ \\
 & $\mathcal{E}_4$ &  $(0.014277689766976964, 0.5424284291298985)$ & $(-0.07027749578541109, 1.6895283608200573)$ \\
\cline{1-4}
\multirow{3}{*}{$m_0$} & \textcolor{blue}{$\mathcal{F}$} & $(0, 0.6138576372995270328)$ & $(0, 1.7746305435327241844)$ \\
 & $\mathcal{F}_3$ & $(-0.016582744680756933, 0.5424600163486574)$ & $(0.08158057491289805, 1.688861170203973)$ \\
 & $\mathcal{F}_4$ & $(0.016582744680756933, 0.5424600163486574)$ & $(-0.08158057491289805, 1.688861170203973)$\\
\cline{1-4}
\multirow{2}{*}{2} &$\mathcal{G}_3$ & $(-0.13477940502391195, 0.5485425650707881)$ & $(0.5904621819032521, 1.5460132316491686)$ \\
 & $\mathcal{G}_4$ & $(0.13477940502391195, 0.5485425650707881)$ & $(-0.5904621819032521, 1.5460132316491686)$ \\
\cline{1-4}
\multirow{2}{*}{$+\infty$} & $\mathcal{H}_3$ & $(-0.23430343925991237, 0.5533288328036454)$ & $(0.8620451062243932, 1.3456025507794094)$ \\
 & $\mathcal{H}_4$ & $(0.23430343925991237, 0.5533288328036454)$ & $(-0.8620451062243932, 1.3456025507794094)$ \\
\bottomrule
\end{tabular}
\caption{Numerical values of $q_3$ and $q_4$ derived from the interval arithmetic and the Krawczyk operator with different $m$ corresponding to the central configurations in Figure \ref{bifurcation2}. We always set $m_1=m_2=1$, $m_3=m_4=m$, and $q_1=(-1,0),q_2=(1,0)$.}
\label{values}
\end{table}
As mentioned in the introduction section, the bifurcations of the concave central configurations with two pairs of equal masses in the whole planar 4-body configuration space, have already been considered in \cite{rusu2016} (Sec. 5.3). 
With one pair set to be 1, and the other pair $m$, the limiting cases, namely, when$m$ goes to $0$ or $\infty$, have already been discussed in \cite{corbera2014}, with five non-collision ones labeled $s_1-s_5$ and three collision ones labeled $sc_1-sc_3$.
To better characterize the bifurcations, we fix the symmetry axis along the perpendicular bisector of $r_{12}$, i.e., coinciding with the $y$-axis, and set $m_1=m_2=1$, $ m_3=m_4=m$, and $q_1=(-1,0),q_2=(1,0)$. With the specific symmetry axis, this framework provides a clear and complete picture of the bifurcations for both symmetric and asymmetric cases, based on part of the numerical analysis in \cites{corbera2014,rusu2016}, as well as the rigorous numerical solutions obtain from the interval arithmetic and the Krawczyk operator, as shown in Figure \ref{bifurcation2} and Table \ref{values}. 
We note in advance that the equal mass
cases $\mathcal{E}_2,\mathcal{E}_3$, and $\mathcal{E}_4$ share the same convex hull,  corresponding to the isosceles triangle central configuration found in \cite{albouy1996}, and their bases are $r_{12}, r_{24}$ and  $r_{14}$ respectively, as shown in Figure \ref{bifurcation2}. In other words, both $\mathcal{E}_3$ and $\mathcal{E}_4$ possess $Z_2$ symmetry, but they are not symmetric with respect to this fixed, specific symmetry axis under our previous settings. Furthermore, by perturbing from $\mathcal{E}_3$ and $\mathcal{E}_4$, we find that most central configurations in these two branches do not possess any symmetry.
That's the main reason we put them in the asymmetric part.
With $m$ increasing from $0$ to $+\infty$, and with the same notations $\tilde m_\ast\approx0.992299447752385347449845$ and $\tilde m_{\ast\ast}\approx0.997294013195487928197522$ used in \cite{rusu2016}, two bifurcation points are identified: at $\mathcal{P}_1$ when $m=\tilde m_\ast$, asymmetric configurations emerge, while at $\mathcal{F}$ when $m=m_0=1/\tilde m_{\ast\ast}$, symmetric configurations cease to exist. Both bifurcation points correspond to those found in \cite{rusu2016}. 
The limiting cases $\mathcal{A}_1$ and $\mathcal{A}_2$ with $m=0$ correpond to $s_1$ and $sc_2$ in \cite{corbera2014}, as well as $\mathcal{H}_3$ and $\mathcal{H}_4$ with $m=\infty$ (or equivalently $m_3=m_4=1$ and $m_1=m_2=0$), correpond to $s_5$ in \cite{corbera2014} when considering the symmetry.

\begin{rema}
The above analysis reveals that a degenerate critical point (such as $\mathcal{P}_1$ in Figure \ref{bifurcation2}) may become non-degenerate within a reduced subspace, causing the bifurcation to disappear when symmetry is taken into account. In contrast, the point $\mathcal{F}$  continues to act as a bifurcation point in this specific subspace.
\end{rema}

\section*{Acknowledgements}
We sincerely thank the anonymous referees for their thorough reading and for their insightful comments and constructive suggestions, all of which have helped us significantly improve the quality of this paper.



\end{document}